\documentclass[reqno,11pt]{article}
\usepackage{amsmath}
\usepackage{amsfonts}
\usepackage{amssymb}
\usepackage{amsthm}
\usepackage{mathrsfs}
\usepackage{enumerate}
\usepackage{color}

\newtheorem{theorem}{Theorem}

\newtheorem{lemma}[theorem]{Lemma}

\newcommand{\NN}{\mathbb{N}}

\newcommand{\ZZ}{\mathbb{Z}}
\newcommand{\CC}{\mathbb{C}}
\newcommand{\mcV}{\mathcal{V}}              
\newcommand{\mcE}{\mathcal{E}}
\newcommand{\mcT}{\mathcal{T}}
\newcommand{\g}{\mathfrak{g}}     

\newcommand{\norm}[1]{\left\Vert#1\right\Vert}    
\newcommand{\floor}[1]{\left\lfloor#1\right\rfloor}        
\newcommand{\fractional}[1]{\left\lbrace#1\right\rbrace}        
\DeclareMathOperator{\Prob}{Pr}              
\newcommand{\expect}{\mathbb{E}}             
\DeclareMathOperator{\adde}{E}               

\begin{document}
\title{Undecidability of Polynomial Inequalities in Subset Densities and Additive Energies}
%
%
\author{Yaqiao Li\footnote{Shenzhen University of Advanced Technology, Shenzhen, China, liyaqiao@suat-sz.edu.cn}}

\maketitle              
\begin{abstract}
Many results in extremal graph theory can be formulated as certain polynomial inequalities in graph homomorphism densities. Answering fundamental questions raised by Lov{\'a}sz, Szegedy and Razborov, Hatami and Norine proved that determining the validity of an arbitrary such polynomial inequality in graph homomorphism densities is undecidable. We observe that many results in additive combinatorics can also be formulated as polynomial inequalities in subset's density and its variants. Based on techniques introduced in Hatami and Norine, together with algebraic and graph construction and Fourier analysis, we prove similarly two theorems of undecidability, thus showing that establishing such polynomial inequalities in additive combinatorics are inherently difficult in their full generality. 

\end{abstract}

\section{Introduction}

A central theme in additive combinatorics \cite{TV2006additive}  is  the interplay between subset density 
and additive structures, the most studied being arithmetic progressions.
A long line of work in this direction culminated in Szemer{\'e}di's theorem, which states that any subset of natural numbers with sufficiently high density contains a long arithmetic progression. 
A variety of remarkable proofs of Szemer{\'e}di's theorem have since been discovered, and have been called by Tao \cite{tao_ICM2005dichotomy} a ``Rosetta stone'' connecting various fields ranging from ergodic theory to Fourier analysis.

Let $A \subseteq G$   be a subset of a finite abelian group $G$, and $\alpha(A) = |A|/|G|$ denotes its density. One way to measure additive structure is to consider the size of sumset
    $A+B = \{a+b: a\in A, b \in B\}$.
The ratio $\sigma(A) = \frac{|A + A|}{|A|}$ is called the \emph{doubling constant} of  $A$.
For example, $\sigma(A) = O(1)$  when $A$ is an arithmetic progression, while $\sigma(A)$ is large when $A$ is random. Another key notion to quantify the additive structure of $A$ is its \emph{additive energy}  $\adde(A)$  defined to be the number of additive quadruples in $A$,
\begin{equation}    \label{eq:def_addEnergy}
    \adde(A) = \left| \{(a_1, a_2, a_3, a_4) \in A^4 : a_1 + a_2 = a_3 + a_4 \} \right|.
\end{equation}
It is easy to see that
    $|A|^2 \le \adde(A) \le |A|^3$.
The higher additive energy, the ``more'' additive structures there are. Additive energy   has been studied in various settings \cite{AE_bloom2018additive,AE_de2023additive,AE_shao2024additive,AE_goh2024entropic}. Below are some general results that concern with subset densities and additive energies. 

\begin{itemize}
    \item The Kneser's theorem \cite{nathanson1996additive}, generalizing the Cauchy-Davenport theorem, states that
        $|A + B| \geq |A| + |B| - |H|$,
    where $A, B \subseteq G$ and $H$ is the stabilizer  of $A+B$. This can be rephrased as
    \begin{equation}    \label{eq:Kneser}
        \alpha(A+B) - \alpha(A) - \alpha(B) + \alpha(H) \ge 0, \quad \forall\ A, B \subseteq G, \quad \forall\ G.
    \end{equation}

    \item The Pl{\"u}nnecke-Ruzsa inequality states that
    $|rB - sB| \le c^{r+s} |A|$
    where
        $c = |A+B|/|A|$,
    and $rB-sB$ denotes the set $B+\cdots+B - B - \cdots - B$ with $r$ $B$'s in the sum and $s$ $B$'s in the difference. This can also be equivalently reformulated as 
    \begin{equation}    \label{eq:Plunnecke-Ruzsa}
        \alpha(A+B)^{r+s} - \alpha(A)^{r+s-1} \cdot \alpha(rB-sB) \ge 0, \quad \forall\ A, B \subseteq G, \quad \forall\ G.
    \end{equation}

    \item The doubling constant and additive energy are closely related. 
    Let 
    \begin{equation}    \label{eq:def_r_Ax}
        r_A(x) = |\{(a_1, a_2) \in A^2 : a_1 + a_2 = x\}|,
    \end{equation}
    which counts how many pairs from $A$ sum to $x$. By double counting we have
        $|A|^2 = \sum_{x \in A+A} r_A(x)$.
    By Cauchy-Schwarz,
    \[
        |A|^4 = \left( \sum_{x \in A+A} r_A(x) \right)^2
        \le \left( \sum_{x \in A+A} r_A(x)^2 \right) \cdot \left( \sum_{x \in A+A} 1^2 \right)
        = \adde(A) \cdot |A+A|.
    \]
    Rewriting gives
        $\adde(A)/|A|^3 \ge 1/\sigma(A)$,
    often described as small doubling implies high additive energy, see a related inequality in \cite{katz2010additive}.
    Dividing both sides with $|G|^4$, we get
    \begin{equation}    \label{eq:poly-energy-vs-doubling}
        \frac{\adde(A)}{|G|^3} \cdot \alpha(A+A) - \alpha(A)^4 \ge 0, \quad \forall\ A \subseteq G, \quad \forall\ G.
    \end{equation}
    
\end{itemize}

The formulation \eqref{eq:Kneser}, \eqref{eq:Plunnecke-Ruzsa} and \eqref{eq:poly-energy-vs-doubling} demonstrate that many results in additive combinatorics can be stated  as the non-negativity of certain polynomials in subset's densities and additive energies of related subsets, and suitable (sometimes repeated) application of Cauchy-Schwarz inequality is often useful for deriving these results, see examples in \cite{gowers2017quantitative}. Other deep results in additive combinatorics such as Freiman’s theorem, Balog-Szemer{\'e}di-Gowers theorem, and the polynomial Freiman-Ruzsa conjecture are also in similar spirit, see a recent breakthrough in \cite{PFR,liao2024improved}.

Szemer{\'e}di originally proved his theorem for arithmetic progressions by developing a far-reaching theorem called Szemer{\'e}di regularity lemma, which is a major step toward the modern study of dense graph limits developed by Lov{\'a}sz et al and the flag algebra by Razborov, in which the graph homomorphism density plays a key role. The homomorphism density of a (small) graph $H$ in a graph  $G$, denoted by $t(H,G)$, is the probability that a random mapping (not necessarily injective) from the vertices of $H$ to the vertices of $G$ maps every edge of $H$ to an edge of $G$. Many important results in extremal graph theory can be described using  polynomial inequalities in homomorphism densities. For example,  Goodman’s theorem \cite{goodman1959sets}, which generalizes the
classical Mantel-Tur{\'a}n theorem, says that 
    $t(K_3, G) - 2t(K_2, G)^2 + t(K_2, G) \ge 0$
holds for every graph $G$, where $K_2$ and $K_3$ denote the edge graph and the triangle graph, respectively. As another example, Sidorenko conjectures that
    $t(H, G) - t(K_2, G)^{|E(H)|} \ge 0$
holds for every bipartite graph $H$ and every graph $G$, where $E(H)$ denotes the edge set of graph $H$.

The above instances show that understanding what polynomial inequalities hold is of central importance in either additive combinatorics or extremal graph theory.  From the perspective of computational complexity, it is therefore a natural and fundamental question to understand the difficulty of establishing such  polynomial inequalities involving the homomorphism density in extremal graph theory, or involving the subset densities and its variants in additive combinatorics. The first question in different forms concerning extremal graph theory  has been formally asked by Lov{\'a}sz, Szegedy and Razborov, and was answered in negative in a breakthrough by Hatami and Norine \cite{hatami2011undecidability}. Specifically, they showed that the problem of determining the validity of a
polynomial inequality between homomorphism densities  is undecidable. Recently, their results have been extended in \cite{blekherman2020simple,blekherman2024undecidability,chen2024undecidability}. In this paper, we consider the second question concerning additive combinatorics. Adapting the techniques in \cite{hatami2011undecidability}, together with algebraic and graph construction and Fourier analysis, we prove similarly two theorems of undecidability, thus showing that establishing such inequalities in additive combinatorics are inherently difficult in their full generality.


We need some notation to state the theorems. Let 
\[
    L = \{f_1(g_1, \ldots, g_k), \ldots, f_d(g_1, \ldots, g_k)\}
\]
be a formal system of linear forms of variables $g_1, \ldots, g_k$. Recall $A \subseteq G$ where $G$ denotes a finite abelian group. We use $L \in A$ to denote that $f_i(g_1,\ldots,g_k) \in A$ for every $i$. By an abuse of notation, we define
\begin{equation} \label{def:t-L-A}
    t(L, A) := \Prob[L \in A] = \Prob[f_1(g_1, \ldots, g_k) \in A, \cdots, f_d(g_1, \ldots, g_k) \in A],
\end{equation}
where $(g_1, \ldots, g_k) \in G^k$ is sampled uniformly at random. Note that we have used the notation $t(\cdot, \cdot)$ to denote both the graph homomorphism density and the probability of $L\in A$. The specific meaning of $t(\cdot, \cdot)$ should be clear from the context. Below are some examples of  $t(L,A)$.
\begin{enumerate}[(i)]
    \item $L=\{f(g) = g\}$. Then, 
        $t(L, A) = \Prob[g \in A] = |A|/|G| = \alpha(A)$
    is the density of $A$.

    \item $L=\{f_1(g,h) = g, f_2(g,h)=h, f_3(g,h) = g+h\}$. Then, 
        $t(L, A) = \Prob[g \in A, h\in A, g+h\in A] = \left( \sum_{x \in A} r_A(x) \right)/ |G|^2$, where $r_A(x)$ is as defined in \eqref{eq:def_r_Ax}.
    
    \item $L=\{f_1(g,h,k) = g, f_2(g,h,k) = h, f_3(g,h,k)=k, f_4(g,h,k) = g+h-k\}$. Then,  
        $t(L, A) = \Prob[g \in A, h \in A, k\in A, g+h-k\in A] =  \adde(A)/|G|^3$,
    i.e., it is a normalized version of the additive energy of $A$ appearing in \eqref{eq:poly-energy-vs-doubling}.
\end{enumerate}

These examples show that, depending on the system of linear forms $L$, the probability $t(L,A)$ can express both subset density and additive energy, and other more complicated variants. 

Let $L, L'$ be two systems of linear forms. We interpret the formal product $L\cdot L'$ as 
    $t(L\cdot L', A) = t(L,A) \cdot t(L',A)$.
In this way, we may consider general \emph{quantum} system of linear forms where we allow formal product of systems of linear forms.

\begin{theorem}  \label{thm:general-undecidable}
The following problem is undecidable.
\begin{itemize}
  \item INSTANCE: Two positive integers $m, k$, quantum systems of linear forms $L_1, \ldots, L_m$ on $k$ variables $g_1, \ldots, g_k$, and integers $a_1, \ldots, a_m$.
  \item QUESTION: Does the inequality $a_1 t(L_1, A) + \cdots + a_m t(L_m, A) \ge 0$ hold for every subset $A \subseteq G$ and every finite abelian group $G$?
\end{itemize}
\end{theorem}

Note that the polynomial $\sum_i a_i t(L_i, A)$ is linear in the variables $t(L_i,A)$, which includes not only $\alpha(A)$ and $\adde(A)$, but also other quantities as we discussed earlier.  The second theorem concerns more specifically to only subsets densities and additive energies. 

\begin{theorem} \label{thm:density-Energy-undecidable}
The following problem is undecidable.
\begin{itemize}
  \item INSTANCE: A positive integer $k \ge 2$, and a polynomial $q(x_1, \ldots, x_k, y_1, \ldots, y_k) \in \ZZ[x_1, \ldots, x_k,  y_1, \ldots, y_k]$.
  \item QUESTION: Does the inequality $q(\alpha(A_1), \ldots, \alpha(A_k), \adde(A_1), \ldots, \adde(A_k)) \ge 0$ hold for all $A_i \subseteq G_i$ and all finite abelian groups $G_i$, for $i=1,\ldots, k$?
\end{itemize}
\end{theorem}

These theorems are ultimately based on the classical Matiyasevich’s theorem of undecidability \cite{matiyasevivc2003enumerable}, that is, given a multivariate polynomial with integer coefficients, the problem of determining whether it is non-negative for every assignment of natural numbers to its variables is undecidable. The undecidability already occurs for a finite number of variables in Theorem \ref{thm:general-undecidable} and Theorem \ref{thm:density-Energy-undecidable}, in fact $k=9$ suffices by Jones \cite{jones1982universal_9unknown}, see also \cite{gasarch2021hilbert,sun2021further}.
Technically, we need to construct certain maps that build correspondence between densities and natural numbers, with respect to related polynomials in question.

\section{Proof of The First Theorem}

\subsection{Two Lemmas}


Let 
    $I_t = [1-\frac{1}{t}, 1-\frac{1}{t+1})$
be sub-intervals of $[0,1]$ for 
    $t = 1, 2, 3, \ldots$.
Define 
\[
    h(x) := \frac{3 t^2 - t - 2}{t(t+1)} x - \frac{2(t-1)}{t+1}, \qquad x \in I_t.
\]
Let $h(1) = 1$. Note that $h(x)$ is piecewise linear on $[0,1]$, and $h(1-\frac{1}{t}) = \frac{(t-1)(t-2)}{t^2}$. The function $h(x)$ comes from Bollob{\'a}s lower bound of the triangle homomorphism density in terms of the edge density,  specifically, the following lower bound between homomorphism densities holds for every graph $G$,
\begin{equation}    \label{eq:Bollobas}
    t(K_3,G) \ge h(t(K_2,G)).
\end{equation}
see detail in \cite[Section 5]{hatami2011undecidability}.
Let
\begin{equation}    \label{eq:R_for_tLA}
    R := \{(x,y) \in [0,1]^2  : y \ge h(x)\}.
\end{equation}
Then, Bollob{\'a}s lower bound implies that
    $(t(K_2,G), t(K_3,G)) \in R$.

The first lemma is a consequence of Lemma 5.1 and Lemma 5.4 in \cite{hatami2011undecidability}.
\begin{lemma}  \label{lem:q-undecidable}
The following problem is undecidable.
\begin{itemize}
\item INSTANCE: A positive integer $k \ge 6$, and a polynomial $q(x_1, \ldots, x_k, y_1, \ldots, y_k) \\ \in \ZZ[x_1, \ldots, x_k, y_1, \ldots, y_k]$.

\item QUESTION: Does the inequality $q(x_1, \ldots, x_k, y_1, \ldots, y_k) \ge 0$ hold for all $(x_1, \ldots, x_k, y_1, \ldots, y_k)$ with $(x_i, y_i) \in R$ for every $1 \le i \le k$?
\end{itemize}
\end{lemma}

Let  $\g = (g_1, \ldots, g_k)$ denote a  $k$-tuple of formal variables. Define a   system of linear forms 
\begin{equation}   \label{eqn:mcL}
    L(\g) := \{\neg (k+1)g_1\} \cup \left( \bigcup_{p=1}^{k+2} \bigcup_{j=2}^k \{p(g_j - j g_1)\}\right).
\end{equation}
Here, the notation $\neg (k+1)g_1$ means that when we consider $L \in A$ we are asking for $(k+1)g_1 \not\in A$. Define two related systems
\begin{align}
    M(\g) &:= L(\g) \cup \{g_1, \ldots, g_k\}, \label{eqn:mcM} \\
    L(\g, j, z) &:= L(g_1, \ldots, g_{j-1}, z, g_{j+1}, \ldots, g_k), \quad j=1,\ldots, k. \label{eqn:mcL-z}
\end{align}

\begin{lemma}  \label{lem:pinpoint}
Let $S = \{0, 1, \ldots, k\} \subseteq G = \ZZ_{(k+1)^2}$, let $\g = (g_1, \ldots, g_k) \in G^k$, then
\begin{itemize}
  \item if $L(\g) \in S$, then $g_j = j g_1$ and $g_j \neq 0$ for every $j=1,\ldots,k$;
  \item if $M(\g) \in S$, then $g_j = j$ for every $j=1,\ldots,k$.
\end{itemize}
\end{lemma}

\begin{proof}
The condition $(k+1)g_1 \notin S$ implies $g_1 \neq 0$. For each $\delta_j = g_j - j g_1$, since we require $\delta_j, 2 \delta_j, \ldots, (k+2)\delta_j$ all lie in $S$, but $S$ contains only $k+1$ distinct elements, by the pigeonhole principle two of them must be equal and hence we have
    $q_j \delta_j = 0$ for some $1 \le q_j \le k+1$.
Since $\delta_j \in S = \{0,1,\ldots,k\}$, we then must have
    $\delta_j = 0$, i.e., $g_j = j g_1$.
The second claim for $M(\g) \in S$ follows from the extra constraints $g_j = jg_1 \in S$ for every $j$. 
\end{proof}

\subsection{Proof of Theorem \ref{thm:general-undecidable}}

Let $G$ denote a finite abelian group. Given two subsets $B, C \subseteq G$, define an associated \emph{directed} graph $U$ where $V(U) = B$, and for $b_1, b_2 \in B$, there is a directed edge $(b_1, b_2) \in E(U)$ if and only if $b_1 - b_2 \in C$.

Let $A \subseteq G$, let $\g = (g_1, \ldots, g_k) \in G^k$. Define
\[
    \mcV_j(\g,z) := M(\g) \cup L(\g, j, z), \quad j=1,\ldots, k.
\]
Define 
\begin{align}   \label{eq:B_C}
    B_j (\g) &= \{z \in G : \mcV_j(\g,z)  \in A\} \subseteq G,  \\
    C_j(\g) &= (B_j(\g) \cap A) - g_j  \subseteq G.
\end{align}
Then construct the associated directed graph $U_j(\g)$ from $B_j(\g)$ and $C_j(\g)$ accordingly.  For $j=1, \ldots, k$, define 
\[
    \mcE_j(\g, z,z') := M(\g) \cup L(\g, j, z) \cup L(\g, j, z') \cup L(\g, j, g_j+ z - z') \cup \{g_j+ z - z'\}.
\]
Intuitively, the system of linear forms $\mcE_j$ is to count the edges in the directed graph $U_j(\g)$. Define $\mcT_j$ similar to $\mcE_j$ but to count the triangles, i.e., 
\begin{align*}
    \mcT_j(\g, z, z', z'') 
    &:= M(\g) \cup L(\g, j, z) \cup L(\g, j, z') \cup L(\g, j, z'')  \\
    & \cup L(\g, j, g_j+ z - z') \cup \{g_j+ z - z'\} \\
    & \cup L(\g, j, g_j+ z' - z'') \cup \{g_j+ z' - z''\} \\
    & \cup L(\g, j, g_j+ z'' - z) \cup \{g_j+ z'' - z\}.
\end{align*}

For the polynomial $q$ as in Lemma \ref{lem:q-undecidable}, define a related polynomial $q^*$ with integral coefficients as follows
\[
    q^*(v_1, \ldots, v_k, e_1, \ldots, e_k, t_1, \ldots, t_k) := q\left(\frac{e_1}{v_1^2}, \ldots, \frac{e_k}{v_k^2}, \frac{t_1}{v_1^3}, \ldots,\frac{t_k}{v_k^3} \right) \prod_{j=1}^k v_j^{3\deg(q)}.
\]
Finally, define the quantum system of linear forms
\[
    \psi(q^*) := q^*(\mcV_1,\ldots,\mcV_k, \mcE_1,\ldots, \mcE_k, \mcT_1, \ldots, \mcT_k).
\]

\begin{proof}[Proof of Theorem \ref{thm:general-undecidable}]
    By Lemma \ref{lem:q-undecidable}, it suffices to show that $q(x,y) \ge 0$ for $(x,y) \in R^k$ is equivalent to $t(\psi(q^*), A) \ge 0$ for every $A \subseteq G$ and for every $G$. 

    We first show 
        $q(x,y) \ge 0$ implies $t(\psi(q^*), A) \ge 0$. 
    Note that for those $\g \in G^k$ such that $M(\g) \notin A$, we have 
        $t(\mcV_j(\g,z), A) = t(\mcE_j(\g,z,z'), A) = t(\mcT_j(\g,z,z', z''), A)=0$. 
    On the other hand, for those $\g$ such that $M(\g) \in A$, we have
    \begin{equation}    \label{eq:density_to_homdensity}
        t(K_2, U_j(\g)) = \frac{t(\mcE_j(\g, z,z'), A)}{t(\mcV_j(\g, z), A)^2},
        \quad
        t(K_3, U_j(\g)) = \frac{t(\mcT_j(\g, z,z',z''), A)}{t(\mcV_j(\g, z), A)^3}.
    \end{equation}
    To verify the equality, as an example, we have
    \begin{align*}
        t(\mcE_j(\g,z,z'), A) 
        &= \Prob[g_j + z-z' \in B_j(\g) \cap A, z \in B_j(\g), z' \in B_j(\g)]  \\
        &= \Prob[z-z' \in C_j(\g) \ |\  z, z' \in B_j(\g)] \cdot \Prob[z\in B_j(\g)]^2  \\
        &= t(K_2, U_j(\g)) \cdot t(\mcV_j(\g, z), A)^2.
    \end{align*}
    One can verify the other equality similarly. This implies,
    \begin{equation}    \label{eq:compute_tqstar}
        t(\psi(q^*), A) =
        \begin{cases}
        q(t(K_2, U_1(\g)), \ldots, t(K_2, U_k(\g)), &\ \\
        \quad t(K_3, U_1(\g)), \ldots, t(K_3, U_k(\g))) &\ \\
        \quad \cdot \prod_{j=1}^k t(\mcV_j(\g,z), A)^{3\deg(q)}, &\quad M(\g) \in A; \\
        0, &\quad M(\g) \notin A.
        \end{cases}
    \end{equation}
    Therefore, since by Bollob{\'a}s lower bound we have 
        $(t(K_2,G), t(K_3,G)) \in R$
    holds for every graph $G$, the claim follows.

    Next, we show $q(x,y) < 0$ for some $(x,y) \in R^k$ implies $t(\psi(q^*), A) < 0$ for some $A\subseteq G$ for some finite abelian group $G$. 
    By the analysis in \cite[Lemma 5.4]{hatami2011undecidability}, we know that there in fact exists some 
        $(x^*_j = 1-\frac{1}{n_j}, y^*_j=2 x_j^2 - x_j) \in R$
    in which $n_j\in \NN$ for every $j \in [k]$, 
    such that 
        $q(x^*, y^*) < 0$.

    Let $H$ be a finite abelian group with subgroups $H_1, \ldots, H_k$ satisfying 
        $\frac{|H_j|}{|H|} = \frac{1}{n_j}$ for every $j \in [k]$. 
    Note that such group and subgroups exist by, e.g., considering direct product of cyclic groups.
    In addition, let 
        $H_0 = \emptyset$. 
    Let 
        $G = \ZZ_{(k+1)^2} \times H$, 
    and
        $S=\{0,1, \ldots, k\} \subseteq \ZZ_{(k+1)^2}$. 
    Consider a subset $A \subseteq G$ defined as follows
    \[
        A := \bigcup_{j\in S} (j \times (H-H_j)).
    \]
    By Lemma \ref{lem:pinpoint},  if $\g = (g_1, \ldots, g_k) \in G$ satisfies $M(\g) \in A$, then, 
        $g_j = (j, h) \in j \times (H-H_j)$.
    By the definition \eqref{eq:B_C}, we have for $j=1, \ldots, k$,
    \begin{equation}    \label{eq:compute_B_C}
        B_j(\g) = j \times H, \quad C_j(\g) = (j \times (H-H_j)) - g_j = 0 \times ((H-H_j) - h).
    \end{equation}
    As an example, we show $B_1(\g) = 1 \times H$. We already know for every $j$, the first component in $g_j \in G$ equals to $j$. Since for $B_1(\g)$, its element 
        $z = (z_1, z_2) \in B_1(\g) \subseteq G$
    plays the role of $g_1$ in the constraints given by the linear forms $L$, we will have
        $((k+1)z_1,(k+1)z_2) \not\in A$.
    Since $H-H_0 = H$, 
        $(k+1)z_2 \in H-H_0$
    always holds, hence, we must have
        $(k+1)z_1 \neq 0$,
    which implies
        $z_1 \neq 0$.
    Furthermore, the other constraints says that
        $p(g_j - j z) \in A$
    for every $j\ge 2$ and for $p=1, \ldots, k+2$.
    If we focus on $z_1$, we would have
        $p(j - j z_1) \in S$
    for every $j\ge 2$ and for $p=1, \ldots, k+2$. 
    This implies 
        $z_1 = 1$.
    But now we have
        $(k+1)z_1 = k+1 \not\in S$, 
    hence, the value $z_2$ could be arbitrary. In other words,
        $B_1(\g) = 1 \times H$.
    The rest can be verified similarly.
    
    The condition \eqref{eq:compute_B_C} implies that if $z,z' \in B_j(\g)$, then we have
        $z = (j, a)$ and $z' = (j,b)$
    for some $a,b \in H$, and
        $z-z' = (0, a-b)$. 
    Hence,
    \begin{align*}
        \frac{t(\mcE_j(\g, z, z'), A)}{t(\mcV_j(\g,z), A)^2} = t(K_2, U_j(\g)) &= \Prob[z - z' \in C_j(\g) \ |\ z, z' \in B_j(\g)] \\
        &= \Prob[a - b \in H-H_j - h \ |\ a, b \in H] \\
        &= \frac{|H-H_j|}{|H|} = 1-\frac{1}{n_j} = x^*_j.
    \end{align*}
    Similarly, we have
        $\frac{t(\mcT_j(\g,z,z',z''), A)}{t(\mcV_j(\g,z), A)^3} 
        = t(K_3, U_j(\g))
        = 2(1-\frac{1}{n_j})^2 - (1-\frac{1}{n_j}) = y^*_j$.
    Hence, by \eqref{eq:compute_tqstar}
    \[
        t(\psi(q^*), A) = q(x^*, y^*) \cdot \prod_{j=1}^k t(\mcV_j(\g,z), A)^{3\deg(q)}< 0.
    \]
\end{proof}


\section{Proof of The Second Theorem}   \label{sec:Pf_Thm2}

The proof of Theorem \ref{thm:general-undecidable} relies on the Bollob{\'a}s lower bound \eqref{eq:Bollobas} of the triangle homomorphism density in terms of the edge density. Now, to study polynomial inequalities involving subset's density and additive energy, it would be useful to establish a similar bound. 

\subsection{An Upper Bound of Additive Energy}  \label{sec:uppbound_addE}

We review some notation of Fourier analysis on a finite abelian group $G$. For every $\xi \in G$, let $\chi_\xi : G \to \CC$ be the characters on $G$ defined as $\chi_\xi(x) = e^{2\pi i \xi x}$. Given two functions $f, g \in L^2(G)$, their convolution $f*g$ is defined as usual $(f*g)(x) := \expect_{y \in G} f(x-y)g(y)$  where the expectation is  taken over $y$ uniformly distributed in $G$. The Fourier transform $\hat{f}: G \to \CC$ of $f$ is defined as $\hat{f}(\xi) := \expect_{x\in G} f(x) \overline{\chi_\xi(x)}$. The $L_2$ norm of $f$ is defined as $\norm{f} = (\expect_{x\in G} |f(x)|^2)^{1/2}$.

For a subset $A \subseteq G$, let $A(x) := 1_A (x)$ be the indicator function for subset $A$. Recall the notation $r_A (x)$ as defined in \eqref{eq:def_r_Ax}. By definition, we have 
    $\frac{r_A(x)}{|G|} = (A*A)(x)$. 
For the additive energy, we have
    $\adde (A) = \sum_{x \in G} r_A (x)^2$.
Hence,
\begin{align*}
    \adde (A) 
    &= \sum_{x \in G} r_A (x)^2 = |G|^2 \sum_{x \in G} (A*A)(x)^2 = |G|^3 \expect_{x \in G} (A*A)(x)^2 \\
    &= |G|^3 \norm{(A*A)(x)}^2.
\end{align*}
Therefore, if we normalize $\adde(A)$ we get $\frac{\adde(A)}{|G|^3} = \norm{(A*A)(x)}^2$. From now on, when we use $\adde(A)$ we mean this normalized version. We now establish the following lemma which bounds the additive energy of a subset $A \subseteq G$ by its density. 

\begin{lemma}  \label{lem:addenergy-upperbound} 
Let $G$ be a finite abelian group, let $A\subseteq G$ be a non-empty subset with density $\alpha = \frac{|A|}{|G|}$, then the (normalized) additive energy $\adde(A)$ satisfies
\begin{equation}   \label{eqn:addenergy-upperbound}
    \adde(A) \le \alpha^3 - \alpha^4 \left(\fractional{\frac{1}{\alpha}} - \fractional{\frac{1}{\alpha}}^2 \right),
\end{equation}  
where 
    $\fractional{\frac{1}{\alpha}} = \frac{1}{\alpha} - \floor{\frac{1}{\alpha}}$ 
denotes the fractional part of $\frac{1}{\alpha}$.
In particular, $\adde(A) \le \alpha^3$ when $\frac{1}{\alpha} \in \NN$, i.e., when $|A|$ divides $|G|$.
\end{lemma}

\begin{proof}
Using $A(x) = 1_A (x)$, we have $\norm{A(x)}^2 = \alpha$, hence by Parseval's identity,
\begin{equation}\label{eqn:sum-of-Fourier-term}
    \sum_{\xi \in G} |\hat{A}(\xi)|^2 = \norm{A(x)}^2 = \alpha,
\end{equation} 
in which 
    $|\hat{A}(\xi)| = \frac{1}{|G|} \left| \sum_{x \in G} A(x) \overline{\chi_\xi(x)} \right| \le \frac{|A|}{|G|}= \alpha$, 
hence,
\begin{equation} \label{eqn:Fourier-term-bound}
    |\hat{A}(\xi)|^2 \le \alpha^2, \quad \forall\ \xi \in G.
\end{equation} 
Since $\adde(A) = \norm{(A*A)(x)}^2$, applying Parseval's identity another time we have
\begin{equation} \label{eqn:addenergy-A}
    \adde(A) = \norm{(A*A)(x)}^2 = \sum_{\xi \in G} |\widehat{A*A}(\xi)|^2 = \sum_{\xi \in G} |\hat{A}(\xi) \cdot \hat{A}(\xi)|^2 = \sum_{\xi \in G} |\hat{A}(\xi)|^4.
\end{equation} 
To get an upper bound of $\adde(A)$ is equivalent to maximize $\adde(A)$ in the form of \eqref{eqn:addenergy-A} under the constraints \eqref{eqn:sum-of-Fourier-term} and \eqref{eqn:Fourier-term-bound}. The Karush–Kuhn–Tucker conditions imply that $\adde(A)$ will be maximized when there are as many $|\hat{A}(\xi)| = \alpha$ as possible, which gives 
\begin{align*} 
    \adde(A) = \sum_{\xi \in G} |\hat{A}(\xi)|^4
	&\le \floor{\frac{1}{\alpha}} \cdot \alpha^4 + \left( \alpha -  \floor{\frac{1}{\alpha}} \cdot \alpha^2 \right)^2 \\
	&= \floor{\frac{1}{\alpha}} \cdot \alpha^4 + \left( \frac{1}{\alpha} -  \floor{\frac{1}{\alpha}} \right)^2 \cdot \alpha^4 \\
	&= \floor{\frac{1}{\alpha}} \cdot \alpha^4 + \fractional{\frac{1}{\alpha}}^2 \cdot \alpha^4.
\end{align*} 
This simplifies to \eqref{eqn:addenergy-upperbound} after a direct calculation.
\end{proof}

Consider $A = \{0\} \subseteq \ZZ_n$, we find that $\alpha(A) = \frac{1}{n}$ and $\adde(A) = \frac{1}{n^3}$, hence the upper bound is tight when $\frac{1}{\alpha} \in \NN$.

\subsection{Proof of Theorem \ref{thm:density-Energy-undecidable}} \label{sec:Pfdetail_Thm2}

We first prove a simple lemma. Let 
    $S= \{\frac{1}{n}: n\in \NN \} \subseteq (0, 1]$.

\begin{lemma}   \label{lem:und_p_on_Sk}
The following problem is undecidable.
\begin{itemize}
  \item INSTANCE: A positive integer $k \ge 2$, and a polynomial $p(x) = p(x_1, \ldots, x_k) \in \ZZ[x_1, \ldots, x_k]$.
  \item QUESTION: Does $p(x_1, \ldots, x_k) \ge 0$ hold for every $(x_1, \ldots, x_k) \in S^k$?
\end{itemize}
\end{lemma}

\begin{proof}
Let $q(x_1, \ldots, x_k) \in \ZZ[x_1, \ldots, x_k]$, and let
\[
p(x_1, \ldots, x_k) = \left( \prod_{i=1}^k x_i^{\deg q}\right) q(\frac{1}{x_1}, \ldots, \frac{1}{x_k}).
\]
Then $p(x_1, \ldots, x_k) \in \ZZ[x_1, \ldots, x_k]$, and clearly $q(x_1, \ldots, x_k) \ge 0$ for $(x_1, \ldots, x_k) \in \NN^k$ is equivalent to $p(\frac{1}{x_1}, \ldots, \frac{1}{x_k}) \ge 0$ for $(\frac{1}{x_1}, \ldots, \frac{1}{x_k}) \in S^k$. But the former is undecidable by Matiyasevich's theorem, hence the latter is undecidable too.
\end{proof}

We are now ready to prove Theorem \ref{thm:density-Energy-undecidable}.

\begin{proof}[Proof of Theorem \ref{thm:density-Energy-undecidable}]
Let 
    $p(x_1, \ldots, x_k) \in \ZZ[x_1, \ldots, x_k]$. 
Choose
\begin{align*}
    M &= \max \Bigg\{1,
    30\max\left\{\norm{\left( \frac{\partial p}{\partial x_1}, \ldots, \frac{\partial p}{\partial x_k} \right)}_\infty: x\in [0,1]^k \right\}, \\
    &3\max\left\{ \norm{\left(\frac{\partial^2 p}{\partial x_1^2}, \ldots, \frac{\partial^2 p}{\partial x_k^2} \right)}_\infty: x\in [0,1]^k \right\}
    \Bigg\} > 0.
\end{align*}

Let $g(\alpha) = \alpha^3$. Let 
    $h(\alpha)= \alpha^3 - \alpha^4 (\fractional{\frac{1}{\alpha}} - \fractional{\frac{1}{\alpha}}^2)$ 
for 
    $\alpha \in (0, 1]$,
where $\fractional{\frac{1}{\alpha}} = \frac{1}{\alpha} - \floor{\frac{1}{\alpha}}$ denotes the fractional part of $\frac{1}{\alpha}$.
Let $h(0)=0$. 
Let 
    $\delta(\alpha) = g(\alpha) - h(\alpha) \ge 0$. 
Note that
    $\delta(\alpha) = 0$ 
or equivalently $h(\alpha) = g(\alpha)$ if $\alpha \in S$.
Define a region
\begin{equation}    \label{eq:R_for_Thm2}
    R := \{(\alpha,\beta)\in [0,1]^2: \beta \le h(\alpha)\}.
\end{equation}

Define
\[
    q(x_1, \ldots, x_k, y_1, \ldots, y_k) := p(x_1, \ldots, x_k) + M \sum_{i=1}^k (g(x_i) - y_i).
\]
We also use the abbreviation $(x, y) = (x_1, \ldots, x_k, y_1, \ldots, y_k)$.

We first show that the following are equivalent: 
(i) $p(x) \ge 0$ for all $x \in S^k$, 
(ii) $q(x,y) \ge 0$ for all $(x, y) \in R^k$.

To show (ii) implies (i), note that if $x \in S^k$, i.e., if $x_i \in S$ we have $(x_i,g(x_i)) \in R$, hence letting $y_i = g(x_i)$, we have
    $p(x) = q(x,y) \ge 0$
as desired.
    
Next we show (i) implies (ii). Consider an auxiliary function 
\[
    \tilde{q}(x_1, \ldots, x_k) := p(x_1, \ldots, x_k) + M \sum_{i=1}^k \delta(x_i).
\]
Since $M > 0$ and when $(x_i,y_i) \in R$ we have $g(x_i) - y_i \ge \delta(x_i)$, hence    
    $q(x,y) \ge \tilde{q}(x)$ holds for $(x,y) \in R^k$.
Hence, it suffices to show $\tilde{q}(x) \ge 0$ for $(x,y) \in R^k$. 
Let 
    $I_t = [\frac{1}{1+t}, \frac{1}{t}]$, $t\in \NN$.
For $\alpha \in I_t$, we have 
    $\fractional{\frac{1}{\alpha}} = \frac{1}{\alpha} - t$.
Hence,
\[
    \delta(\alpha) = -t(t+1)\alpha^4 + (2t+1)\alpha^3 - \alpha^2, \quad \alpha \in I_t.
\]
We have 
    $\delta'(\alpha) = -4t(t+1)\alpha^3 + 3(2t+1)\alpha^2 - 2\alpha$ 
and 
    $\delta''(\alpha) = -12t(t+1)\alpha^2 + 6(2t+1)\alpha - 2$ for $\alpha \in I_t$. 
With some calculation using calculus, we have
\[
    \delta'(\alpha) \ge \frac{1}{20} > 0, \quad \alpha \in \left[ \frac{1}{3}, \frac{2}{5} \right] \cup \left[\frac{1}{2}, \frac{7}{10} \right],
\]
and
\[
    \delta''(\alpha) \le -\frac{1}{2} < 0, \quad \alpha \in \cup_{t\ge 3} I_t \cup \left[\frac{2}{5}, 1/2 \right] \cup \left[\frac{7}{10}, 1 \right].
\]
By our choice of $M$, we have
\[
    \frac{\partial \tilde{q}}{\partial x_i} = \frac{\partial p}{\partial x_i} + M \delta'(x_i) > 0, \quad x_i \in \left[\frac{1}{3}, \frac{2}{5} \right] \cup \left[ \frac{1}{2}, \frac{7}{10} \right],
\]
and
\[
\frac{\partial^2 \tilde{q}}{\partial x_i^2} = \frac{\partial^2 p}{\partial x_i^2} + M \delta''(x_i) < 0, \quad x_i \in \cup_{t\ge 3} I_t \cup \left[ \frac{2}{5}, 1/2 \right] \cup \left[ \frac{7}{10}, 1 \right].
\]
Hence, we see that $\tilde{q} (x)$ achieves its locally minimal value when $x \in S^k$, in which case $\tilde{q}(x) = p(x) \ge 0$ as needed.

By Lemma \ref{lem:addenergy-upperbound}, 
    $(\alpha(A_i), \adde(A_i)) \in R$.
Hence, if there exists $A_i\subseteq G_i$ for $i=1, \ldots, k$ such that
    $q(\alpha(A_1), \ldots, \alpha(A_k), \adde(A_1), \ldots, \adde(A_k)) < 0$
this would imply
    $q(x,y) < 0$ for some $(x,y) \in R^k$.
On the other hand, note that the proof of the equivalence also shows that the local minima of $q(x,y)$ are achieved at $x \in S^k$. Hence, if $q(x,y) < 0$ for some $(x,y) \in R^k$ there will be $x \in S^k$ (and $y_i = x_i^3$) to certify this. The remark after the proof of Lemma \ref{lem:addenergy-upperbound} shows that there are $A_i \subseteq G_i$ such that $(\alpha(A_i), \adde(A_i))$ achieves these values, hence we will have
    $q(\alpha(A_1), \ldots, \alpha(A_k), \adde(A_1), \ldots, \adde(A_k)) < 0$.
Therefore, the undecidability result of Theorem \ref{thm:density-Energy-undecidable} follows from Lemma \ref{lem:und_p_on_Sk}.
\end{proof}

{\bf Acknowledgement} The author thanks Hamed Hatami for generously sharing his insights, and referees for useful comments.

%
%
%
 \bibliographystyle{plain}
 \bibliography{mybib}

\end{document}